\documentclass[a4paper,12pt]{article}
\usepackage[divide={2.45cm,*,2.45cm}]{geometry}              % See geometry.pdf to learn the layout options. There are lots.
\usepackage{graphicx}
\usepackage{amssymb}
\usepackage{epstopdf}
\usepackage{amsthm}
\usepackage{extpfeil}

\newcommand {\pl}[1]{\ensuremath{\frac{\partial}{\partial#1}}}
\newcommand{\dbar}{\ensuremath{\bar{\partial}}}
\newcommand{\dop}[1]{\ensuremath{\partial_{#1}}}
\newcommand{\dbop}[1]{\ensuremath{\partial_{\bar{#1}}}}

\newcommand{\E}{\ensuremath{\mathcal{E}}}

\renewcommand{\thanks}[1]{\footnote{#1}} % Use this for footnotes

\newcommand{\be}{\begin{equation}}
\newcommand{\bea}{\begin{eqnarray}}
\newcommand{\eea}{\end{eqnarray}} \newcommand{\ee}{\end{equation}}
 
 \def\ba{\begin{eqnarray}}
\def\ea{\end{eqnarray}}

\def\E{{\cal E}}

\def\det{{\rm det}}

\def\log{\,{\rm log}\,}

\def\k{\kappa}

\def\[{{\bf [}}
\def\]{{\bf ]}}

\def\pl{\partial}

\begin{document}
\theoremstyle{plain}
\newtheorem{Lemma}{Lemma}[section]
\newtheorem{prop}{Proposition}[section]
\newtheorem{Theorem}{Theorem}[section]
\newtheorem{corollary}{Corollary}[section]
\theoremstyle{definition}
\newtheorem{definition}{Definition}[section]
\newtheorem{question}{Question}
\newtheorem{example}{Example}
\theoremstyle{remark}
\newtheorem*{remark}{Remark}
\newtheorem{claim}{Claim}
\newtheorem{proposition}{Proposition}
\newtheorem{lemma}{Lemma}

\begin{centering}
 
\textup{\LARGE\bf{On the convergence of the Sasaki-Ricci flow}}

\vspace{8mm}

\textnormal{\large{Tristan C. Collins and Adam Jacob}}
\vspace{5mm}
\begin{abstract}
{\small Given a Sasaki manifold $S$, we prove the Sasaki-Ricci flow converges exponentially fast to a Sasaki-Einstein metric if one exists, provided the automorphism group of the transverse holomorphic structure is trivial. }

\end{abstract}

\end{centering}

\begin{normalsize}
\section{Introduction}

The geometry of Sasaki manifolds has recently garnered a great deal of interest due to its role in the AdS/CFT correspondence of theoretical physics.  Moreover, Sasakian geometry can be viewed as both a specialization and generalization of K\"ahler geometry, and so is of independent interest.  In particular, it is an important problem to determine necessary and sufficient conditions for the existence of Sasaki-Einstein metrics.  In K\"ahler geometry, a great deal of progress has been made towards establishing necessary and sufficient conditions for the existence of both K\"ahler-Einstein and constant scalar curvature K\"ahler metrics, though a complete solution is as of yet unavailable.  It was conjectured by Yau $\cite{Y1}$ that the existence of canonical K\"ahler metrics is equivalent to some GIT notion of stability, in analog with the known correspondence between slope stability and the existence of Hermitian-Einstein metrics on holomorphic vector bundls $\cite{Don1},\cite{UY}$.  Since any K\"ahler manifold gives rise to a Sasaki manifold, one expects that the solution of the Sasaki-Einstein problem should be intimately related to stability in some appropriately generalized GIT sense.  

One approach to the K\"ahler-Einstein problem which has seen some success in generating sufficient conditions for the existence is the K\"ahler-Ricci flow (see e.g. \cite{PS06},\cite{PSSW}). However, it is still unknown how these conditions for convergence of the flow, such as a positive lower bound for the smallest eigenvalue of the Laplacian and the vanishing of the Futaki invariant, relate to other forms of  stability, such as Chow-Mumford stability, K-stability \cite{T97}, \cite{Don4}, uniform K-stability \cite{Sze},
slope-stability \cite{RT}, and b-stability \cite{Don2}.  An unpublished result of Perelman claims that, when a K\"ahler-Einstein metric exists, the K\"ahler-Ricci flow will converge to it $\cite{TZ}$, and so the existence problem for K\"ahler-Einstein metrics is equivalent to establishing necessary and sufficient conditions under which the Ricci flow converges.  There has been some success using algebraic notions of stability to obtain convergence of the flow, though usually in conjunction with some assumption on the boundedness of the Riemann tensor along the flow $\cite{T}$.  A flow approach to the Sasaki-Einstein problem was recently introduced by Smoczyk, Wang, and Zhang in the paper $\cite{SWZ}$, in which they generalized the results of Cao $\cite{Cao}$.  The deep estimates of Perelman for the K\"ahler-Ricci flow $\cite{ST}$ were generalized to the Sasaki setting in $\cite{T1}$ and then used to prove a uniform Sobolev inequality along the Sasaki-Ricci flow in $\cite{T2}$, generalizing results of $\cite{Z1},\cite{Z2}$.  In $\cite{T3}$ necessary and sufficient conditions for the convergence of the Sasaki-Ricci flow were developed in the spirit of Phong, Song, Sturm and Weinkove $\cite{PSSW}$, in addition to Zhang $\cite{Z3}$, focusing on the dimension of the space of the holomorphic global sections of a certain sheaf $\mathcal{E}$.  In this paper we apply these developments to extend the theorem of Perelman to the Sasaki-Ricci flow.  Let $Aut^0(S)$ denote the identity component of the automorphism  group of the transverse holomorphic structure of $(S,\xi,\eta,g_0)$. We then prove the following theorem:

\begin{Theorem}
\label{maintheorem}
If $(S,\xi, \eta, g_{0})$ admits a Sasaki-Einstein metric $g_{SE}$ and $Aut^{0}(S)=\{ e\}$, then the Sasaki-Ricci flow converges exponentially fast to a Sasaki-Einstein metric.  
\end{Theorem}

Our development follows the ideas of $\cite{PS06}$ closely.  We utilize an inequality of Moser-Trudinger type, proved in the Sasaki case by Zhang $\cite{XZ}$ and in the K\"ahler Einstein case by Tian $\cite{T97}$, and subsequently improved by Tian and Zhu $\cite{TZ2}$, and Phong, Song, Sturm and Weinkove $\cite{PSSW2}$. For fixed contact form $\eta$ on $S$, consider the following functionals defined on all basic potentials $\phi$:
\bea
J_{\eta}(\phi)&:=&\frac{1}{V}\int_0^1\int_S\dot\phi_t\,(d\mu-d\mu_\phi)\,dt,\nonumber\\
F_{\eta}(\phi)&:=&J_\eta(\phi)-\frac{1}{V}\int_S\phi\, d\mu-{\rm log}\left(\frac{1}{V}\int_Xe^{h-\phi}d\mu\right),\nonumber
\eea
where $\phi_t$ is any path with $\phi_0=c$ and $\phi_1=\phi$ and $h$ is the transverse Ricci-potential associated to $\eta$. We need the following inequality:
\begin{Theorem}[$\cite{XZ}$, Theorem 6.1]
\label{MTI}
If $(S,\xi, \eta, g_{0})$ admits a Sasaki-Einstein metric $g_{SE}$ and $Aut^{0}(S)=0$, then there exists positive constants $A$,$B$ such that following inequality holds for all potentials:
\be
F_{\eta_{0}}(\phi)\geq A\,J_{\eta_{0}}(\phi)-B.\nonumber
\ee
\end{Theorem}

With this inequality in hand, and the uniform Sobolev inequality in $\cite{T2}$, an application of the results of $\cite{T3}$ allow us to obtain exponential convergence of the Sasaki-Ricci flow to a Sasaki-Einstein metric. The outline of this paper is as follows.  In the Section 2 we introduce some of the basic objects in our development, including the Sasaki-Ricci flow. Then, in Section 3, we complete the proof of Theorem $\ref{maintheorem}$ using Theorem $\ref{MTI}$ and parabolic estimates along the flow.

\medskip
\begin{centering}
{\bf Acknowledgements}
\end{centering}
\medskip

First and foremost, the authors would like to thank their thesis advisor, D.H. Phong, for all his guidance and support during the process of writing this paper. The authors also thank Valentino Tosatti for much encouragement and some helpful suggestions. This research was funded in part by the National Science Foundation, Grant No. DMS-07-57372. 

\section{Preliminaries}
\label{prelim}

We assume that the reader is familiar with the basic aspects of Sasaki geometry; for a good introduction, we refer to $\cite{BG},\cite{Sp}$.  Let $(S,g_{0})$ be a Sasaki manifold of dimension $m=2n+1$ with Reeb vector field $\xi$ and contact 1-form $\eta$.  Let $L_{\xi}$ be the line subbundle of $TS$ generated by $\xi$ and let $D$ denote the contact subbundle $D:=\ker \eta = L_{\xi}^{\perp} \subset TS $.  Recall that the basic functions are those functions which are invariant under the flow generated by the Reeb field.  The set of Sasaki metrics compatible with the Sasaki structure on $S$ is parametrized by the space of basic potentials:
\begin{equation*}
P(S,\eta)=\{\phi\in C^\infty_B(S)|\,\eta_\phi:=\eta+d^c_B\phi {\rm\,\, is\, \,a \,\,contact \,\,1\,form}\}.\nonumber
\end{equation*}
Here $d_{B} = \partial_{B} + \dbar_{B}$ is the basic de Rham differential. From now on, we shall assume that the basic first Chern class $c_{1}^{B}(S) >0$ and that $c_{1}(D) =0$, so that there is no obstruction to assuming that $\k\, d\eta \in c_{1}^{B}(S)$ for some constant $\k$ $\cite{FOW}$. If $g^T$ is the transverse metric on $D$, then we say $g^T$ is transverse Einstein if it satisfies $Ric^T=\k\,g^T$. It is not hard to check that given a transverse Einstein metric $g^T$, the corresponding metric $g$ on $S$ is Sasaki-Einstein if and only if $\k=2n+2$. Therefore, for the rest of the paper, we assume the normalization constant $\k$ is equal to $2n+2$. The normalized Sasaki-Ricci flow is defined by:
\begin{equation*}
\dot g(t)^{T} = \kappa g^{T}(t)-Ric^{T}(t).
\end{equation*}
Long time existence of the Sasaki-Ricci flow in the canonical case was established in $\cite{SWZ}$.  Since the Sasaki-Ricci flow preserves the transverse K\"ahler class, we can write this flow as a transverse parabolic Monge-Amp\`ere equation on the potentials:
\begin{equation}\label{PCMA}
\dot\phi = \log \det((g_{0})_{\bar{k}l}^{T}+\partial_{l}\partial_{\bar{k}}\phi)- \log \det((g_{0})_{\bar{k}l}^{T})+\kappa\phi -h(0).
\end{equation}
Here $h(0)$ is the transverse Ricci potential of $\eta_0$, defined by
 \be
 Ric_0^T-\kappa g_0^T=\frac{i}{2}\pl_B\bar\pl_Bh(0),\nonumber
 \ee
which we always normalize so that $\int_Se^{h(0)}d\mu_0=1$. Such a potential exists by the transverse $\partial\dbar$-lemma in $\cite{EK}$, and is a basic function. Let $h(t)$ be the evolving transverse Ricci potential, then we have:
\begin{equation*}
\dop{j}\dbop{k}\dot{\phi} = \dot{g}^{T}_{\bar{k}j} = -R^{T}_{\bar{k}j} +\kappa g^{T}_{\bar{k}j} = \dop{j}\dbop{k}h(t).
\end{equation*}
From this equation we see $\phi$ evolves by $\dot{\phi}(t) = h(t) + c(t)$, for $c(t)$ depending only on time. We can use the function $c(t)$ to adjust the initial value $\phi(0)$.  We shall always assume that :
\begin{equation}\label{initial condition}
\phi(0)=c_{0} := \int_{0}^{\infty} e^{-t}\|\nabla \dot{\phi} \|_{L^{2}}^{2} dt + \frac{1}{V}\int_{S}h(0) d\mu_0.
\end{equation}
The constant $c_{0}$ plays an important role in proving convergence of the flow, and is discussed in $\cite{T3}$.  The transverse Ricci potential $h(t)$ evolves by:
\begin{equation*}
\dot{h} = \Delta h+h +a(t),
\end{equation*}
for some basic function $a(t)$ depending only on $t$, which we fix by requiring $\int e^{-h}\,d \mu_{\phi} = V$.  Under these assumptions, we have the following theorem, which was first proven in $\cite{T1}$:
\begin{Theorem}
\label{Perelman}
Let $g^T(t)$ be a solution to the Sasaki-Ricci flow on $S$. Let $h(t)\in C^\infty_B(S)$ be the evolving transverse Ricci potential. Then there exists a uniform constant $C$, depending only on $g(0)$, so that
\be
|R^T(g(t))|+|h(t)|_{C^1(g(t))}<C.\nonumber
\ee
\end{Theorem}

Here $R^T$ is the transverse scalar curvature of $g^T$. We also have the following uniform Sobolev inequality along the flow, as proven in $\cite{T2}$;
\begin{Theorem}
\label{Sobolev}
Let $g^T(t)$ be a solution to the Sasaki-Ricci flow on $S$, and let $S$ have real dimension $m=2n+1$. Then for every $v\in W^{1,2}_B(S)$, we have:
\be
\left(\int_Sv^{\frac{2m}{m-2}}\,d\mu_{\phi}\right)^{\frac{m-2}{m}}\leq C\int_S(|\nabla v|^2+v^2)\,d\mu_\phi,\nonumber
\ee
where $C$ only depends on $g^T(0)$ and $m$.
\end{Theorem}
Here we note the importance of Theorem~\ref{Sobolev} is that both integrals are with respect to the evolving volume form. We now cite the following proposition from $\cite{T3}$:
\begin{prop}\label{uniform Yau}
Let $(S, \xi, \eta_{0}, \Phi, g_{0})$ be a compact Sasaki manifold with $\kappa [\frac{1}{2}d\eta_{0}]_{B} =c_{1}^{B}(S)$ for any constant $\kappa$.  Consider the Sasaki-Ricci flow defined by~(\ref{PCMA}), with $\phi(0) = c_{0}$.  Then we have the a priori estimates
\begin{equation*}
\sup_{t\geq 0} \|\phi\|_{C^{0}} \leq A_{0} <\infty \iff \sup_{t\geq 0} \| \phi \|_{C^{k}} \leq A_{k} <\infty \quad \forall k \in \mathbb{N}
\end{equation*}
\end{prop}
In particular, establishing a $C^{0}$ bound suffices to obtain convergence along a subsequence of the Sasaki-Ricci flow.  Once convergence along a subsequence is established, one can argue as in the proof of Lemma 9.4 from $\cite{T3}$ to obtain convergence of the whole sequence to a Sasaki-Einstein metric.  Our goal then, is to prove that the Moser-Trudinger inequality implies a uniform $C^{0}$ bound for the potential $\phi$ along the Sasaki-Ricci flow.

\subsection{Transverse foliate vector fields on Sasaki manifolds}
In the K\"ahler theory, the presence of global holomorphic vector fields plays a critical role in the existence theory for canonical K\"ahler metrics.  In the Sasaki setting, we are lead to study the following sheaf:

\begin{definition}
On an open subset $U\subset S$, let $\Xi(U)$ be the Lie algebra of smooth vector fields on $U$ and let $\mathcal{N}_{\xi}(U)$ be the normalizer of the Reeb field in $\Xi (U)$,
\begin{equation*}
\mathcal{N}_{\xi}(U) = \{ X \in \Xi(U) : [X,\xi] \in L_{\xi}\}.
\end{equation*}
We define a sheaf $\mathcal{E}$ on $S$ by
\begin{equation*}
\mathcal{E}(U) := N_{\xi}(U)/L_{\xi}.
\end{equation*}
The sheaf $\E$ will be referred to as the sheaf of transverse foliate vector fields.
\end{definition}

The sheaf $\E$ inherits a holomorphic structure from the transverse complex structure, and has a well-defined $\dbar$ operator, as discussed in $\cite{T3}$.  The global holomorphic sections of this sheaf are related to the holomorphic, Hamiltonian vector fields on $S$ (see $\cite{FOW}$), by the following proposition:

\begin{prop}[Proposition 5.3 from $\cite{T3}$]\label{kernel prop}
We define the space $H^{0}(\E^{1,0})$, which we refer to as the space global holomorphic sections of the sheaf of transverse foliate vector fields, by
\begin{equation*}
H^{0}(\E^{1,0}):=Ker \bar{\partial}_{\E} \big|_{\E^{(1,0)}}.
\end{equation*}
\begin{enumerate}
\item[\textbullet] $H^{0}(\E^{1,0})$ has the structure of a finite dimensional Lie algebra over $\mathbb{C}$.
\item[\textbullet] $H^{0}(\E^{1,0})$ is isomorphic as a Lie algebra to the Lie algebra of holomorphic, Hamiltonian vectorfields on $S$.  
\item[\textbullet] The space  $H^{0}(\E^{1,0})$ depends only on the complex structure $J$ on the cone, and the Reeb field $\xi$, and the transverse holomorphic structure.  In particular, $\dim H^{0}(\E^{1,0})$ is invariant along the Sasaki-Ricci flow.
\end{enumerate}
\end{prop}

When the metric $\kappa g_{SE} \in c_{1}^{B}(S)$ is Sasaki-Einstein, the Lie algebra $H^{0}(\E^{1,0})$ is isomorphic to the kernel of $\Delta_{B} +\k$ by Theorem 5.1 in $\cite{FOW}$.  This presents a significant difficulty, as $\Delta_{B} +\k$ is precisely the linearization of the Monge-Amp\`ere operator at $g_{SE}$, and its kernel is an obstruction to solving the Monge-Amp\`ere equation using the backwards method of continuity  $\cite{NS}$.  Because of the fundamental role played by the method of continuity in establishing the Moser-Trudinger inequality, we are forced to assume that $H^{0}(\E^{1,0})$ is empty.  At the end of the paper we shall see how to relax this assumption.

Finally we relate the sheaf $\E$ with  Aut$(S)$, the automorphism group of the transverse holomorphic structure. This group is defined to be  biholomorphic automorphisms of the K\"ahler cone $(C(S),J)$ which commute with the holomorphic flow generated by $\xi-i\,J(\xi)$. Let Aut$(S)^0$ be the connected component at the identity, and consider the subgroup Stab$(g_{SE})\subset$Aut$(S)^0$ consisting of automorphisms which fix the Sasaki-Einstein metric $g_{SE}$. Then, if we look at the orbit ${\cal O}$ of $g_{SE}$ under Aut$(S)^0$, in $\cite{NS}$ it is shown that ${\cal O}$ has the natural topology of the homogenous space ${\cal O}\cong$Aut$(S)^0/$Stab$(g_{SE})$. Furthermore, they obtain:
\be
T_{g_{SE}}{\cal O}\cong H^{0}(\E^{1,0}).\nonumber
\ee
Thus if Aut$(S)^0=\{ e\}$, we know $H^{0}(\E^{1,0})=0$, and the backwards Monge-Amp\`ere equation admits a solution.

\subsection{Important Functionals}
\label{functionals}

Here we introduce some functionals which will be important in our development, all of which are defined in $\cite{NS}$. For notational simplicity throughout the paper we denote the volume form on $S$ defined by $\eta$ as $d\mu=(d\eta)^n\wedge\eta$. Given a potential $\phi$, the volume form with respect to $\eta_\phi:=\eta+d^C_B\phi$ is given by $d\mu_\phi=(d\eta_\phi)^n\wedge\eta$. Now, consider the following functionals on $P(S,\eta)$: 
\bea
I_\eta(\phi)&:=&\frac{1}{V}\int_S\phi\,(d\mu-d\mu_\phi)\nonumber\\
J_\eta(\phi)&:=&\frac{1}{V}\int_0^1\int_S\dot\phi_t\,(d\mu-d\mu_{\phi_t})\,dt,\nonumber
\eea
where $\phi_t$ is any path with $\phi_0=c$ and $\phi_1=\phi$. Various forms of these functionals exist (for details see $\cite{NS}$), and we can use these formulations to prove:
\be
\label{JIJ}
\frac{1}{n}J_\eta\leq\frac{1}{n+1}I_\eta\leq J_\eta.
\ee
The time derivatives of these functionals along any path $\phi_t$ can now be computed easily:
\bea
\pl_tI_\eta(\phi_t)&:=&\frac{1}{V}\int_S\dot\phi_t\,(d\mu-d\mu_{\phi_t})-\frac{1}{2V}\int_S\phi_t\,\pl_t\,d\mu_{\phi_t},\nonumber\\
\pl_tJ_\eta(\phi_t)&:=&\frac{1}{V}\int_S\dot\phi_t\,(d\mu-d\mu_{\phi_t}).\nonumber
\eea
Thus, the time derivative of the difference is given by:
\be
\pl_t(I_\eta-J_\eta)(\phi_t)=-\frac{1}{V}\int_S\phi_t\,\pl_t \,d\mu_{\phi_t}.\nonumber
\ee
Next, we consider the following two functionals, which differ only by the last term:
\bea
\label{F function 1}
F_{\eta}^{0}(\phi) &:=& J_{\eta}(\phi) - \frac{1}{V}\int_{S}\phi \,d\mu.\\
 F_\eta(\phi)&:=&J_\eta(\phi)-\frac{1}{V}\int_S\phi\, d\mu-{\rm log}\left(\frac{1}{V}\int_Xe^{h-\k\phi}d\mu\right).\nonumber
 \eea
 Here $h$ is the transverse Ricci potential of $\eta$. Finally we define the transverse K-energy, which once again is defined along any path $\phi_t$ with $\phi_0=c$ and $\phi_1=\phi$:
 \be
 K_\eta(\phi)=-\frac{1}{V}\int_0^1\int_{S}\dot\phi_t(R^T_{\phi_t}-n\k) \,d\mu_{\phi_t}.\nonumber
 \ee

\section{Convergence of the Sasaki-Ricci flow}
Here we use the Moser-Trudinger inequality stated in Theorem $\ref{MTI}$ to show a uniform $C^{0}$ bound for the potential $\phi$ along the Sasaki-Ricci flow. Our first step is to establish some relations between the functionals defined in the previous subsection, and in particular we make explicit use the Sasaki-Ricci flow.

\begin{Lemma}\label{PS lemma 8} There exists constants $C_{1}, C_{2}$, depending only on $g_{0}$, so that if $\phi=\phi(t)$ is evolving along the Sasaki-Ricci flow, we have:
\bea
i)&\qquad&K_{\eta_{0}}(\phi) - F_{\eta_{0}}^{0}(\phi) -\frac{1}{V}\int_{S}\dot{\phi}\,d\mu_{\phi} = C_{1},\nonumber\\
ii)&\qquad&|F_{\eta_{0}}(\phi) - K_{\eta_{0}}(\phi)| + |F^{0}_{\eta_{0}}(\phi) - K_{\eta_{0}}(\phi)| \leq C_{2}\nonumber
\eea
\end{Lemma}
\begin{proof}
We begin with {\it i\,}).  Our first goal is to compute the time derivative of $F^0_{\eta_0}$ along the flow. Using the formula for the variation of $J_{\eta_{0}}$ from Section $\ref{functionals}$, we have:
\begin{equation*}
\pl_tF_{\eta_{0}}^{0}(\phi) = -\frac{1}{V}\int_{S}\dot{\phi(t)}\,d \mu_{\phi}.
\end{equation*}
On the other hand, since $\ddot{\phi} = \Delta_B\dot{\phi} +\dot{\phi}$, we obtain:
\begin{equation*}
\frac{1}{V}\int_{S}\dot{\phi}\,d \mu_{\phi} = \frac{1}{V}\int_{S}\ddot{\phi} \,d \mu_{\phi} = \pl_t\left(\frac{1}{V}\int_{S}\dot{\phi}\,\,d \mu_{\phi}\right) - \frac{1}{V}\int_{S}\dot{\phi}\Delta_{B}\dot{\phi}\,d \mu_{\phi}.
\end{equation*}
Computing the evolution equation for $K_{\eta_{0}}$ we have:
\begin{equation*}
\pl_tK_{\eta_{0}}(\phi)= - \frac{1}{V}\int_{S}\dot{\phi}(R^{T}-n\k)\,d \mu_{\phi} =\frac{1}{V}\int_{S}\dot{\phi}\Delta_{B}\dot{\phi} \,d \mu_{\phi}.
\end{equation*}
Combining the above equations, we obtain:
\begin{equation*}
\pl_tF^{0}_{\eta_{0}}(\phi) = -\pl_t\left(\frac{1}{V}\int_{S}\dot{\phi}\,d \mu_{\phi}\right) +\pl_tK_{\eta_{0}}(\phi),
\end{equation*}
from which {\it i\,}) follows.  To prove {\it ii\,}), observe that, by Theorem $\ref{Perelman}$, $\dot{\phi}$ is uniformly bounded by a constant depending only on $g_{0}$, and so:
\begin{equation*}
|F^{0}_{\eta_{0}}(\phi) - K_{\eta_{0}}(\phi)|< C(g_{0}).
\end{equation*}
To establish the second inequality, we use the definition of $F_{\eta}$ and employ the uniform bound for $\dot{\phi}$ again to obtain:
\begin{equation*}
\begin{aligned}
|F_{\eta_{0}}(\phi) - K_{\eta_{0}}(\phi)| &\leq \left|F^{0}_{\eta_{0}}(\phi)-K_{\eta_{0}}(\phi)\right| + \left|\log\left(\frac{1}{V}\int_{S}e^{h(0)-\k\phi}\,d\mu_0\right)\right|\\
&\leq C(g_{0}) +\left|\log\left(\frac{1}{V}\int_{S}e^{-\dot{\phi}}\,d\mu_{\phi}\right)\right|\\
&\leq  C(g_{0}) + C'(g_{0}),
\end{aligned}
\end{equation*}
where in the second line we used the Sasaki-Ricci flow equation for potentials $\eqref{PCMA}$. This establishes {\it ii\,}).
\end{proof}

\begin{Lemma}\label{PS lemma 9}
There exists a constant $C$ so that the following estimates hold uniformly along the Sasaki-Ricci flow:
\bea
iii)&\qquad&\frac{1}{nV}\int_{S} (-\phi)\,d \mu_{\phi} -C \leq J_{\eta_{0}}(\phi) \leq \frac{1}{V} \int_{S}\phi\,d\mu_0+C\nonumber\\
iv)&\qquad&
\frac{1}{V}\int_{S}\phi\,d\mu_0 \leq \frac{n}{V}\int_{S}(-\phi)\,d \mu_{\phi} - (n+1)K_{\eta_{0}}(\phi) +C.\nonumber
\eea
\end{Lemma}
\begin{proof}
Since the Mabuchi K-energy decreases monotonically along the Sasaki-Ricci flow, the second inequality in Lemma $\ref{PS lemma 8}$ implies that $F_{\eta_{0}}(\phi) \leq C$ uniformly along the flow.  Rearranging equation~(\ref{F function 1}), combined with the upper bound for $F_{\eta_{0}}^{0}(\phi)$ yields the right hand inequality in {\it iii\,}). Now, by definition of the $F^0_{\eta_0}$ we have the following equation:
\begin{equation}\label{F function 2}
F^{0}_{\eta_{0}}(\phi) = -\left[(I_{\eta_{0}}-J_{\eta_{0}})(\phi) + \frac{1}{V}\int_{S}\phi \,d\mu_\phi\right].
\end{equation} 
Rearranging this equation, and applying the upper bound for $F_{\eta_{0}}^{0}(\phi)$ yields:
\begin{equation*}
\frac{1}{V}\int_{S} (-\phi)\,d\mu_\phi-C\leq (I_{\eta_{0}}-J_{\eta_{0}})(\phi).
\end{equation*} 
Applying estimate $\eqref{JIJ}$ establishes the left hand inequality in {\it iii\,}).  To establish {\it iv\,}), we apply Lemma~\ref{PS lemma 8} inequality {\it ii\,}) and equation~(\ref{F function 1}) to obtain:
\begin{equation*}
\frac{1}{V}\int_{S}\phi\,\,d\mu_0 \leq J_{\eta_{0}}(\phi) - K_{\eta_{0}}(\phi)+C \leq \frac{n}{n+1}I_{\eta_{0}}(\phi) - K_{\eta_{0}}(\phi) +C,
\end{equation*}
where the second inequality follows from $\eqref{JIJ}$.  Applying the definition of the functional $I_{\eta_{0}}$ and rearranging terms yields the result.
\end{proof}

The following proposition is a corollary of the uniform Sobolev inequality given in Theorem $\ref{Sobolev}$.

\begin{prop}\label{moser iteration}
Along the Sasaki-Ricci flow we have the following inequality
\begin{equation*}
\text{osc}(\phi) \leq \frac{A}{V}\int_{S}\phi \,d\mu_0 + B,
\end{equation*}
 where $A$ and $B$ are constants depending only on $g_{0}$.
 \end{prop}
 \begin{proof}
 Define the function $f = \max_{S}\phi - \phi +1 \geq 1$.  We now apply the standard Moser iteration technique.  Let $\alpha> 0$ and write:
 \begin{equation*}
 \begin{aligned}
 \int_{S}f^{\alpha+1}(d\eta_{\phi})^{n}\wedge \eta_{0} &\geq \int_{S} f^{\alpha+1}(d\eta_{\phi} - d\eta_{0})\wedge d\eta_{\phi}^{n-1}\wedge \eta_{0}\\ &= -\frac{i}{2}\int_{S}f^{\alpha+1}\partial\dbar f\wedge(d\eta_{\phi})^{n-1}\wedge\eta_{0}.
 \end{aligned}
 \end{equation*}
 We now integrate by parts:
 \begin{equation*}
 \begin{aligned}
 -\frac{i}{2}\int_{S}f^{\alpha+1}\partial\dbar f\wedge(d\eta_{\phi})^{n-1}\wedge\eta_{0} &=  -\frac{i}{2}\int_{S}f^{\alpha+1}d_{B}\dbar f\wedge(d\eta_{\phi})^{n-1}\wedge\eta_{0}\\ &=  \frac{i(\alpha+1)}{2}\int_{S}f^{\alpha}\partial f\wedge\dbar f\wedge(d\eta_{\phi})^{n-1}\wedge\eta_{0}\\&=\frac{i(\alpha+1)}{2(\frac{\alpha}{2}+1)^{2}}\int_{S}\partial( f^{\frac{\alpha}{2}+1})\wedge\dbar (f^{\frac{\alpha}{2}+1})\wedge(d\eta_{\phi})^{n-1}\wedge\eta_{0}.
 \end{aligned}
 \end{equation*}
 Thus, we obtain:
 \begin{equation}\label{equation 8.36}
 \|\nabla(f^{\frac{\alpha}{2}+1})\|_{L^{2}(S,\eta_{\phi})}\leq \frac{n(\frac{\alpha}{2}+1)^{2}}{\alpha+1} \int_{S}f^{\alpha+1}\,d\mu_\phi.
 \end{equation}
 Set $\beta = \frac{2n+1}{2n-1}$ (here $m=2n+1$ is the real dimension of S), and $p=\alpha+2 \geq 2$.  Thus, because we have uniform control of the Sobolev constant along the flow $\cite{T2}$, we obtain:
 \begin{equation*}
 \left[\int_{S}f^{p\beta}\,d \mu_{\phi})\right]^{\frac{1}{\beta}} \leq Cp\int_{S}f^{p}\,d \mu_{\phi}
 \end{equation*}
 for a constant $C$ depending only on $g_{0}$.  Taking $p=2$ and iterating in the usual fashion we obtain:
 \begin{equation*}
 \log\left(\|f\|_{L^{\infty}(S)}\right) \leq \sum_{k=1}^{\infty}\frac{\log(2C\beta^{k})}{2\beta^{k}} + \log\left(\|f\|_{L^{2}(S,\eta_{\phi})}\right) = C_{1} + \log\left(\|f\|_{L^{2}(S,\eta_{\phi})}\right).
 \end{equation*}
 It remains only to bound the $L^{2}$ norm of $f$.  By the Poincar\'e inequality in the appendix of $\cite{T1}$, we have that:
 \begin{equation*}
 \frac{1}{V}\int_{S}f^{2}e^{h}\,d \mu_{\phi} \leq \frac{1}{V}\int_{S}|\nabla f|^{2}e^{h}\,d \mu_{\phi} + \left(\frac{1}{V}\int_{S}f e^{h} d\mu_\phi\right)^{2},
 \end{equation*}
where $h=h(t)$ is the transverse Ricci potential.  Moreover, by the uniform bounds for $h$ along the Sasaki-Ricci flow from Theorem $\ref{Perelman}$, the measures $e^{h} \,d \mu_{\phi}$ and $\,d \mu_{\phi}$ are equivalent.  We obtain:
\begin{equation*}
\begin{aligned}
\|f\|^{2}_{L^{2}(S,\eta_{\phi})} &\leq C\frac{1}{V}\int_{S}|\nabla f|^{2}\,d \mu_{\phi} + C\left( \frac{1}{V}\int_{S}f \,d \mu_{\phi}\right)^{2}\\ &\leq C'\left[ 1+\frac{1}{V} \int_{S}f \,d \mu_{\phi}\right]^{2}
\end{aligned}
\end{equation*}
where the final inequality follows by applying equation~(\ref{equation 8.36}) with $\alpha =0$.  Finally, since $\Delta_{g_{0}}\phi >-n$, a standard argument with the Green's function of $g_{0}$ yields:
\begin{equation*}
\sup_{S}\phi \leq \frac{1}{V}\int_{S}\phi \,d\mu_0 + C''.
\end{equation*}
Moreover, by Lemma~\ref{PS lemma 9} part {\it iii\,}) we have:
\begin{equation*}
\frac{1}{V}\int_{S}(-\phi)\,d \mu_{\phi} \leq \frac{n}{V}\int_{S}\phi \,d\mu_0 +C'''.
\end{equation*}
Applying the definition of $f$, the proposition follows.
%\begin{aligned}
%\frac{1}{V}\int_{S} f(d\eta_{\phi})^{n}\wedge \eta_{0} &= \sup_{S}\phi +\frac{1}{V}\int_{S}(-\phi)(d\eta_{\phi})^{n}\wedge \eta_{0} +1\\ &\leq A\frac{1}{V}\int_{S}\phi(d\eta_{\phi})^{n}\wedge \eta_{0} + B
%\end{aligned}
%\end{equation*}
%where the last inequality follows 
 \end{proof}
 
 Combining Proposition~\ref{moser iteration} with an argument from $\cite{T3}$, we can prove the following corollary:

\begin{corollary}\label{convergence cor}
Let $(S,\eta_{0},\xi, g_{0})$ be a compact Sasaki manifold, with $d\eta_{0} \in c_{1}^{B}(S)$, and consider the Sasaki-Ricci flow with initial value given by $\eqref{initial condition}$. If there exists a constant $C$ with 
\begin{equation}
\sup_{t\in[0,\infty)} \frac{1}{V}\int_{S}\phi \,d\mu_0 \leq C <\infty,
\end{equation}
then the Sasaki-Ricci flow converges exponentially fast in $C^{\infty}$ to a Sasaki-Einstein metric.
\end{corollary}
\begin{proof}
Proposition~\ref{moser iteration} implies that $\text{osc}(\phi)$ is uniformly bounded along the Sasaki-Ricci flow.  Moreover, we have:
\begin{equation*}
1 = \frac{1}{V}\int_{S}\,d \mu_{\phi} = \frac{1}{V}\int_{S}e^{-\k\phi+\dot{\phi}+h(0)}\,d\mu_0.
\end{equation*}
Now, since $\|\dot{\phi}\|_{C^{0}}$ is uniformly controlled along the Sasaki-Ricci flow by Theorem $\ref{Perelman}$, we have
\begin{equation*}
0<C_{1} \leq \frac{1}{V}\int_{S}e^{-\k\phi}\,d\mu_0 \leq C_{2},
\end{equation*}
which easily implies a lower bound for $\sup_{S}\phi$.  Combined with the uniform bound for $\text{osc}(\phi)$ we obtain we obtain a uniform bound for $\|\phi\|_{C^{0}}$.  By Proposition $\ref{uniform Yau}$ we obtain uniform bounds for $\phi$ in $C^{k}(S,g_{0})$.  We can now apply the argument in the proof of Lemma 9.4 in $\cite{T3}$ to obtain the exponential convergence of $\phi$ to a Sasaki-Einstein potential.
\end{proof}

We are now ready to prove our main result:
\begin{proof}[Proof of Theorem $\ref{maintheorem}$]
By assumption, the Moser-Trudinger inequality holds along the Sasaki-Ricci flow.  Since $K_{\eta_{0}}(\phi)$ is decreasing along the flow, by Lemma~\ref{PS lemma 8} we know $F_{\eta_0}$ is bounded from above. It follows from Theorem $\ref{MTI}$, applied with the reference metric equal to $\eta_0$, that $J_{\eta_{0}}$ is uniformly bounded from above. Thus by Lemma~\ref{PS lemma 9} inequality $iii$) we have:
\be
\int_S(-\phi)\,d\mu_\phi\leq C.\nonumber
\ee
 Since $J_{\eta_0} \geq 0$, applying the Moser-Trudinger inequality we know that $F_{\eta_{0}}(\phi)$ is uniformly bounded below.  Then again applying Lemma~\ref{PS lemma 8} we see the Mabuchi K-energy $K_{\eta_{0}}(\phi)$ is uniformly bounded from below. By Lemma~\ref{PS lemma 9}, part {\it ii}) we obtain:
\begin{equation*}
\frac{1}{V}\int_{S}\phi \,d\mu_0 <C.
\end{equation*}
The desired result follows from Corollary~\ref{convergence cor}.
\end{proof}

Here we remark that our result can be easily generalized to a case where Aut$(S)^0\neq 0$. Let $G\subset$ Stab$(g_{SE})$ be a closed subgroup whose centralizer in Stab$(g_{SE})$ is finite. Then, following the proof of Theorem 2 from $\cite{PSSW2}$, the Moser-Trudinger inequality can be extended to all $G$-invariant potentials. Using this fact, the convergence of the Sasaki-Ricci flow as stated in Theorem $\ref{maintheorem}$ works for all $G$-invariant initial Sasaki metrics $g_0$.

 \end{normalsize}

\vspace{10mm}

\begin{centering}

\textnormal{ Department of Mathematics,
Columbia University,
New York, NY 10027\\
ajacob@math.columbia.edu\\
tcollins@math.columbia.edu}

\end{centering}

\end{document}